\documentclass[12pt]{article}

\usepackage{geometry} 
\usepackage{amsthm,amssymb}
\usepackage{amsmath,amsfonts,latexsym}
\usepackage{latexsym}
\usepackage{float}
\usepackage{xcolor}
\usepackage{url}

\def\titlerunning#1{\gdef\titrun{#1}}
\makeatletter
\def\author#1{\gdef\autrun{\def\and{\unskip, }#1}\gdef\@author{#1}}
\def\address#1{{\def\and{\\\hspace*{18pt}}\renewcommand{\thefootnote}{}%
\footnote {#1}}%
\markboth{\autrun}{\titrun}}
\makeatother
\def\email#1{\hspace*{4pt}{\em e-mail}: #1}

\newtheorem{thm}{Theorem}[section]

\newtheorem{cor}[thm]{Corollary}
\theoremstyle{definition}

\newtheorem{example}[thm]{Example}

\def\GF{{\rm GF}}
\def\Cay{{\rm Cay}}

\begin{document}

\titlerunning{}
\title{New constructions of divisible design Cayley graphs}

\author{Dean Crnkovi\' c and Andrea \v Svob}

\maketitle

\address{D. Crnkovi\'{c}, A. \v Svob: Department of Mathematics, University of Rijeka, Radmile Matej\v{c}i\'c 2, 51000 Rijeka, Croatia;
\email{\{deanc,asvob\}@math.uniri.hr} 
\and 
Corresponding author: A. \v Svob
}

\begin{abstract}
Divisible design graphs were introduced in 2011 by Haemers, Kharaghani and Meulenberg. 
Further, divisible design graphs which can be obtained as Cayley graphs were recently studied by Kabanov and Shalaginov. 
In this paper we give new constructions of divisible design Cayley graphs and classify divisible design Cayley graphs on $v \le 27$ vertices.
\end{abstract}

\bigskip

{\bf 2020 Mathematics Subject Classification:} 05B30, 05E18, 05E30.

{\bf Keywords:} divisible design, Cayley graph, regular group.

\section{Introduction}\label{intro}

We assume that the reader is familiar with the basic facts of group theory,
graph theory and design theory. We refer the reader to \cite{bjl, atlas, diestel, r} on terms not defined in this paper.

\bigskip

A graph $\Gamma$ can be interpreted as a design by taking the vertices of
$\Gamma$ as points, and the neighbourhoods of the vertices as blocks.
Such a design is called a neighbourhood design of $\Gamma$. The adjacency
matrix of $\Gamma$ is the incidence matrix of its neighbourhood design.

A $k$-regular graph on $v$ vertices with the property that any two distinct
vertices have exactly $\lambda$ common neighbours is called a $(v,k, \lambda)$-graph
(see \cite{rudvalis}).
The neighbourhood design of a $(v,k, \lambda)$-graph is a symmetric $(v,k, \lambda)$
design. Haemers, Kharaghani and Meulenberg have defined divisible design graphs
(DDGs for short) as a generalization of $(v,k, \lambda)$-graphs (see \cite{ddg}).

An incidence structure with $v$ points and the constant block size $k$ is a 
(group) divisible design with parameters $(v,k, \lambda_1, \lambda_2, m,n)$ 
whenever the point set can be partitioned into $m$ classes of size $n$, 
such that two points from the same class are incident with exactly $\lambda_1$ common blocks, and two points from different classes are 
incident with exactly $\lambda_2$ common blocks. 
Note that a DDG with $m = 1$, $n = 1$, or $\lambda_1 = \lambda_2$ is a $(v, k, \lambda)$-graph. If this is the case, we
call the DDG improper, otherwise it is called proper.
Panasenko and Shalaginov made a list of all proper divisible design graphs with up to 39 vertices \cite{p-s}.
A divisible design $D$ is said to be symmetric (or to have the dual property) if the dual of $D$ is a 
divisible design with the same parameters as $D$. The definition of
a DDG yields the following theorem (see \cite{ddg}).

\begin{thm}
If $D$ is a divisible design graph with parameters 
$(v,k, \lambda_1, \lambda_2, m,n)$ then its 
neighbourhood design is a symmetric divisible design 
$(v,k, \lambda_1, \lambda_2, m,n)$.
\end{thm} 

Conversely, a symmetric divisible design with a polarity with no absolute points is the neighborhood design of a DDG (see \cite{ddg}).

\bigskip

Let $G$ be a group and $S$  a subset of $G$ not containing the identity element of the group, which will be denoted by $e$. The vertices of the Cayley digraph $\Cay(G,S)$ are the elements 
of the group $G$, and its arcs are all the couples $(g,gs)$ with $g \in G$ and $s \in S$. Goryainov, Kabanov, and Shalaginov \cite{kabanov} studied divisible design
Cayley graphs. Further, divisible design Cayley digraphs were studied in \cite{c-h-s}.
In this paper we give some new constructions of divisible design Cayley graphs, and classify such graphs on $v \le 27$ vertices.

\section{Divisible design Cayley graphs}\label{CDDGs}

The following well-known theorem (see \cite{sabidussi}) provides a characterization of Cayley graphs.

\begin{thm}\label{regular}
A graph $G$ is a Cayley graph of a group if and only if $Aut(G)$ contains a regular subgroup. 
\end{thm}

The following characterization of divisible design Cayley graphs is given in \cite{kabanov}.
Note that a Deza graph with parameters $(n,k,b,a)$ is a $k$-regular graph with $n$ vertices in which any two vertices have $a$ or $b$ $(a \le b)$ common neighbours. 

\begin{thm} \label{cosets-DDG}
Let $\Cay(G,S)$ be a Deza graph with parameters $(v,k,b,a)$ and $SS^{-1}=aA+bB+k{e}$, where $A$, $B$ and $\{e\}$ be a partition of G. If either $A \cup \{e\}$ or $B \cup \{e\}$ is a subgroup of $G$, 
then $\Cay(G,S)$ is a DDG and the right cosets of this subgroup give a canonical partition of this graph.  
Conversely, if $\Cay(G,S)$ is a DDG, then the class of its canonical partition which contains the identity of G is a subgroup of G and classes of the canonical partition of DDG
coincides with the cosets of this subgroup.
\end{thm}

\begin{example}\label{v=8}
Let $G= \langle a, b \ | \ a^4=b^2= e, ab=ba \rangle$ be a group isomorphic to $Z_4 \times Z_2$, and let $S= \{a, a^2, a^3, b \}$. Then $S = S^{-1}$ and
$SS^{-1}=2A+0B+4e$, where $A= \{ a,b,ab,a^2,a^3,ba^3 \}$ and $B= \{ b \}$. Obviously, $B \cup \{ e \}$ is a subgroup of $G$, so $\Cay(G,S)$ is a DDG and 
the right cosets of $B$ give a canonical partition of a divisible design Cayley graph with parameters $(8,4,0,2,4,2)$.
This DDG can also be constructed as a Cayley graph from the groups $D_8$ and $E_8$.
\end{example}

In \cite{kabanov}, the authors gave a nice construction of divisible design Cayley graphs from finite fields of order $q^r$, $q$ a prime power and $r>1$ that is based on Theorem \ref{cosets-DDG}.
We found Theorem \ref{cosets-DDG} very useful when constructing divisible design Cayley graphs with the help of a computer. 

\bigskip

In the sequel we give constructions of divisible design Cayley graphs and nonexistence results. 
Throughout the paper we denote by $I_v$, $O_v$ and $J_v$ the identity matrix, the zero-matrix and the all-one matrix of size $v \times v$, respectively.

\subsection{Nonexistence results}\label{nonexistence}

A list of feasible parameters for DDGs on at most 27 vertices is given in \cite{ddg}. In this section we establish nonexistence of some divisible design Cayley graphs with feasible parameters.
The results are obtained using Magma \cite{magma} and GAP \cite{gap-grape}.

\begin{thm} \label{thm-nonex}
The following divisible design Cayley graphs do not exist: 
$$(15,4,0,1,5,3), (20,9,0,4,10,2), (24,6,2,1,3,8), (24,10,6,3,3,8), (27,8,4,2,9,3).$$
\end{thm}

\subsection{Constructions of divisible design Cayley graphs} \label{constructions}

We will denote by $I_v$, $O_v$ and $J_v$ the identity matrix, the zero-matrix and the all-one matrix of size $v \times v$, respectively.
\bigskip

\begin{thm} \label{Kronecker} 
Let $A$ be the incidence matrix of a $(v,k, \lambda)$-graph with a regular automorphism group $G$. Then there exists a divisible design Cayley graph with parameters $(vt,k,\lambda,0,t,v)$.
\end{thm}
\begin{proof}
The Kronecker product $I_t \otimes A$ is the adjacency matrix of a DDG with parameters $(vt,k,\lambda,0,t,v)$.
The direct product of the group $G$ and the cyclic group $Z_t$ acts regularly on this DDG. By Theorem 2, the constructed DDG is a Cayley graph.
\end{proof}

\bigskip

In the proof of Theorem \ref{Kronecker}, instead of $I_t$ one can use any symmetric permutation matrix $P$.

\begin{thm}\label{Con-4.4-hkm}
Let $A$ be the incidence matrix of a $(v,k, \lambda)$-graph with a regular automorphism group $G$. Then there exists a divisible design Cayley graph with parameters $(vn,kn,kn,\lambda n,v,n)$.
\end{thm}
\begin{proof}
The Kronecker product $A \otimes J_n$ is the adjacency matrix of a DDG with parameters $(vn,kn,kn,\lambda n,v,n)$ (see \cite[Construction 4.4]{ddg}).
This DDG admits a regular action of the direct product of the cyclic group $Z_n$ and the group of $G$, so it is a divisible design Cayley graph.
\end{proof}

The strong product of two graphs with adjacency matrices $A$ and $B$ is the graph with adjacency matrix $(A + I) \otimes (B + I) - I$.

\bigskip

\begin{thm} \label{Con-4.10-hkm}
Let $\Gamma$ be a strongly regular graph with parameters $(m, k,\lambda, \lambda +1)$ with a regular automorphism group $G$. 
Then the strong product of $K_2$ with $\Gamma$ is a Cayley DDG with parameters $(2m,2k+1,2k,2 \lambda + 2,m,2)$.
\end{thm}
\begin{proof}
By \cite[Construction 4.10]{ddg}, the strong product of $K_2$ with $\Gamma$ is a DDG with parameters $(2m,2k+1,2k,2 \lambda + 2,m,2)$.
The group $Z_2 \times G$ acts regularly on this DDG. 
\end{proof}

Let $q$ be a prime power. If $q \equiv 3\ (mod\ 4)$ then the set of non-zero squares in $\GF(q)$ forms a difference set in the additive group of $\GF(q)$, and in case $q \equiv 1\ (mod\ 4)$ 
the set of non-zero squares in $\GF(q)$ forms a partial difference set in the additive group of $\GF(q)$. 
The conclusion is  that the case $q \equiv 3\ (mod\ 4)$ yields a Cayley digraph, and the case $q \equiv 1\ (mod\ 4)$ yields a Cayley graph. 
The adjacency matrix of the Cayley digraph obtained in the case $q \equiv 3\ (mod\ 4)$ is the incidence matrix of a symmetric design called a Paley design, which is a Hadamard design with parameters 
$(q,\frac{q-1}{2},\frac{q-3}{4})$. The fact that $-1$ is not a square in the field $\GF(q)$ when $q \equiv 3\ (mod\ 4)$ implies that the incidence matrix of a Paley design is skew. The graph
obtained in the case $q \equiv 1\ (mod\ 4)$ is called a Paley graph, which is a strongly regular graph with parameters $(q,\frac{q-1}{2},\frac{q-5}{4},\frac{q-1}{4})$.

The Paley graphs belong to the family of SRGs with parameters $(v, k,\lambda, \lambda +1)$, so the following corollary is a direct consequence of Theorem \ref{Con-4.10-hkm}.

\begin{cor} \label{Paley-cor} 
Let $q$ be a prime power, $q \equiv 1\ (mod\ 4)$. Then there exists a divisible design Cayley graph with parameters $(2q,q,q-1,\frac{q-1}{2},q,2)$.
\end{cor}

An $m \times m$ matrix $H$ is a Hadamard matrix if every entry is 1 or -1, and $HH^t = mI_m$. A Hadamard matrix $H$ is called graphical if $H$ is symmetric with constant diagonal, 
and regular if all row and column sums are equal (to $\ell$ say). Without loss of generality we assume that a graphical Hadamard matrix has diagonal entries -1. 

\bigskip

Consider a regular graphical Hadamard matrix $H$. It is well known that $\ell^2 = m$ and that $\frac{1}{2} (H + J)$ is the adjacency matrix of a $(m, (m + \ell)/2, (m + 2 \ell)/4)$-graph.

\begin{thm} \label{Thm-3.9-dc-wh}
Let $H$ be a regular graphical Hadamard matrix of order $4u^2$ with
diagonal entries $-1$ and row sum $2u$ ($u$ can be negative), having a regular automorphism group $G$.
Further, let $D$ be the adjacency matrix of an $(n, k', \lambda)$-graph with a regular automorphism group $G_1$.
Replace each entry $-1$ of $H$ by $D$, and each $+1$ by $J_n-D$.
Then we obtain the adjacency matrix $A$ of a Cayley DDG with parameters
$(4nu^2, 2nu^2+u(n-2k'), 4\lambda u^2+u(2u+1)(n-2k'), nu^2+u(n-2k'), 4u^2,n).$
\end{thm}
\begin{proof}
The matrix $A$ is the adjacency matrix of a DDG with parameters $(4nu^2, 2nu^2+u(n-2k'), 4\lambda u^2+u(2u+1)(n-2k'), nu^2+u(n-2k'), 4u^2,n)$ (see \cite[Theorem 3.9]{dc-whh}).
Obviously, $G \otimes G_1$ acts regularly on the DDG obtained.
\end{proof}

In the following theorem we give parameters of divisible design Cayley graphs for $v \le 27$, $0 < \lambda_2 < k$ and $\lambda_1 < k$, whose existence were established using a computer.

\begin{thm}\label{comp}
There exist divisible design Cayley graphs with the following parameters:
\begin{displaymath}   
\begin{tabular}{l l l l}
$(12,5,0,2,6,2)$ \hspace{0.5cm}  & $(12,5,1,2,4,3)$ \hspace{0.5cm}  & $(12,6,2,3,3,4)$ \hspace{0.5cm}  & $(12,7,3,4,4,3)$ \\
$(18,9,6,4,6,3)$ \hspace{0.5cm}  & $(20,7,3,2,4,5)$  \hspace{0.5cm}  & $(20,7,6,2,10,2)$ \hspace{0.5cm} & $(20,13,9,8,4,5)$ \\
$(20,13,12,8,10,2)$ \hspace{0.5cm} & $(24,7,0,2,8,3)$ \hspace{0.5cm} & $(24,8,4,2,4,6)$ \hspace{0.5cm}  & $(24,10,2,4,12,2)$  \\
 $(24,10,3,4,8,3)$ \hspace{0.5cm}  & $(24,14,6,8,12,2)$ \hspace{0.5cm} &$(24,14,7,8,8,3)$ \hspace{0.5cm}  & $(24,16,12,10,4,6)$ \\
$(27,18,9,12,9,3).$                 &                                   &                                  &   \\
\end{tabular}                
\end{displaymath}
\end{thm}

\section{Small parameters} \label{small}

All feasible parameter sets $(v, k, \lambda_1, \lambda_2, m, n)$ for DDGs on at most 27 vertices are given in \cite{ddg}. 
The DDGs on at most 27 vertices were further studied in \cite{dc-whh}, including the walk-regularity of DDGs. 
The only set of parameters for which the existence of DDGs have not been decided in \cite{ddg} and \cite{dc-whh} is (27,16,12,9,9,3).
Recently, Dmitry Panasenko and Leonid Shalaginov \cite{dp} showed that there is no quotient matrix for a putative DDG with parameters (27,16,12,9,9,3), and therefore there is no DDG(27,16,12,9,9,3).
Thereby, the existence for DDGs with parameters $(v, k, \lambda_1, \lambda_2, m, n)$ is decided for all DDGs with $v \le 27$.
In {\rm T}able \ref{table1} we give parameters of proper DDGs with $v\leq 27$, $0<\lambda_2<2k-v$, $\lambda_1<k$, and the number of the divisible design Cayley graphs with the given parameters, 
up to isomorphism ($\#$DDCGs). Constructions of divisible design Cayley graphs with $\lambda_2=0$ and $\lambda_1=k$ are given in Theorem \ref{Kronecker} and Theorem \ref{Con-4.4-hkm}, respectively.

\begin{table}[H]
{\footnotesize
\begin{center}
\begin{tabular}{|rrrrrr|c|ccc|}
\hline
$v$ & $k$ & $\lambda_1$ & $\lambda_2$ & $m$ & $n$ & reference & Cayley & reference & $\#$DDCGs \\[3pt]
\hline
 8 &  4 &  0 &  2 &  4 & 2 &  \cite{ddg} & yes& Example \ref{v=8} &1 \\
10 &  5 &  4 &  2 &  5 & 2 &  \cite{ddg} & yes & Corollary \ref{Paley-cor} &1 \\
12 &  5 &  0 &  2 &  6 & 2 &  \cite{ddg} & yes& Theorem \ref{comp} &1 \\
12 &  5 &  1 &  2 &  4 & 3 &  \cite{ddg} & yes& Theorem \ref{comp} &1 \\
12 &  6 &  2 &  3 &  3 & 4 &  \cite{ddg} & yes& Theorem \ref{comp} &1 \\
12 &  7 &  3 &  4 &  4 & 3 &  \cite{ddg} & yes& Theorem \ref{comp} &1 \\
15 &  4 &  0 &  1 &  5 & 3 &  \cite{ddg} & no& &0 \\
18 &  9 &  6 &  4 &  6 & 3 &  \cite{ddg} & yes& Theorem \ref{comp} &1 \\
18 &  9 &  8 &  4 &  9 & 2 &  \cite{ddg} & yes & Corollary \ref{Paley-cor} &1 \\
20 &  7 &  3 &  2 &  4 & 5 &  \cite{ddg} & yes& Theorem \ref{comp} &1 \\
20 &  7 &  6 &  2 & 10 & 2 &  \cite{ddg} & yes& Theorem \ref{comp} &1 \\
20 &  9 &  0 &  4 & 10 & 2 &  \cite{ddg} & no & &0 \\
20 & 13 &  9 &  8 &  4 & 5 &  \cite{ddg} & yes& Theorem \ref{comp} &1 \\
20 & 13 & 12 &  8 & 10 & 2 &  \cite{ddg} & yes& Theorem \ref{comp} &1 \\
24 &  6 &  2 &  1 &  3 & 8 &  \cite{ddg} & no& &0 \\
24 &  7 &  0 &  2 &  8 & 3 &  \cite{ddg} & yes& Theorem \ref{comp} &1 \\
24 &  8 &  4 &  2 &  4 & 6 &  \cite{ddg} & yes& Theorem \ref{comp} &3 \\
24 & 10 &  2 &  4 & 12 & 2 &  \cite{dc-whh} &yes & Theorem \ref{comp} &3 \\
24 & 10 &  3 &  4 &  8 & 3 &  \cite{dc-whh} & yes& Theorem \ref{comp} &1 \\
24 & 10 &  6 &  3 &  3 & 8 &  \cite{ddg} & no& &0 \\
24 & 14 &  6 &  8 & 12 & 2 &  \cite{dc-whh} & yes& Theorem \ref{comp} &1 \\
24 & 14 &  7 &  8 &  8 & 3 &  \cite{ddg} & yes& Theorem \ref{comp} &1 \\
24 & 16 & 12 & 10 &  4 & 6 &  \cite{ddg} & yes& Theorem \ref{comp} &3 \\
26 & 13 & 12 &  6 & 13 & 2 &  \cite{ddg} & yes & Corollary \ref{Paley-cor} &1 \\
27 &  8 &  4 &  2 &  9 & 3 &  \cite{dp} & no& &0 \\
27 & 18 &  9 & 12 &  9 & 3 &  \cite{ddg} & yes& Theorem \ref{comp} &2 \\
\hline
\end{tabular}
\caption{Proper DDGs with $v\leq 27$, $0<\lambda_2<2k-v$, $\lambda_1<k$.}\label{table1}
\end{center}
}
\end{table}

\begin{table}[H] {\scriptsize
\label{DDDDs}
\begin{center}
\begin{tabular}{ c|c|c}
DDCG & regular groups acting on a DDCG  & $\#$ DDCG
\\
\hline\hline
(8,4,0,2,4,2) &  $D_8$ ($Z_4\times Z_2, E_8$) & 1 \\
\hline
(10,5,4,2,5,2) & $D_{10}$ ($Z_{10}$) & 1 \\
\hline
(12,5,1,2,4,3) & $D_{12}$ ($Z_3: Z_4, Z_{12}, A_4, Z_6\times Z_2$) & 1 \\
(12,5,0,2,6,2) & $A_4$ & 1 \\
(12,6,2,3,3,4) & $A_4$ & 1 \\
(12,7,3,4,4,3) & $D_{12}$ ($Z_3: Z_4, Z_{12}, Z_6\times Z_2$) & 1 \\
\hline
(18,9,6,4,6,3) & $Z_3\times S_3$ ($E_9:Z_2$) & 1 \\
(18,9,8,4,9,2) & $Z_3\times S_3$ ($E_9:Z_2, Z_6\times Z_3$) & 1 \\
\hline
(20,7,3,2,4,5) & $D_{20}$ ($Z_5: Z_4, Z_{20}, Z_{10}\times Z_2$) & 1 \\
(20,7,6,2,10,2) & $Z_5: Z_4$ & 1 \\
(20,13,9,8,4,5) & $D_{20}$ ($Z_5: Z_4, Z_{20}, Z_{10}\times Z_2$) & 1 \\
(20,13,12,8,10,2) & $Z_5: Z_4$ & 1 \\
\hline
(24,7,0,2,8,3) & $S_4$ & 1 \\
(24,8,4,2,4,6) & $Z_4\times S_3$ ($D_{24}$, $(Z_6\times Z_2):Z_2$) & 1 \\ 
							&  $Z_4\times S_3$ ($D_{24}$, $Z_2\times (Z_3:Z_4)$, $(Z_6\times Z_2):Z_2$, & 1 \\ 
							&$Z_{12}\times Z_2$, $Z_3\times D_8$, $S_4$, $Z_2\times A_4$, $Z_2\times Z_2 \times S_3$, $Z_6\times Z_2\times Z_2$) & \\
							& $Z_4\times S_3$ ($D_{24}$, $(Z_6\times Z_2):Z_2$, $Z_2\times (Z_3:Z_4)$, $Z_{12}\times Z_2$, $Z_3\times D_8$ )& 1  \\		
(24,10,2,4,12,2) & $S_4$ & 1 \\
(24,10,3,4,8,3) & $S_4$ & 1 \\
(24,14,6,8,12,2) & $S_4$ & 1 \\
(24,14,7,8,8,3) & $S_4$ & 1 \\
(24,16,12,10,4,6) & $Z_4\times S_3$ ($D_{24}$, $(Z_6\times Z_2):Z_2$, $Z_2\times Z_2 \times S_3$ ) & 1 \\ 
							&  $Z_4\times S_3$ ($D_{24}$, $Z_2\times (Z_3:Z_4)$, $(Z_6\times Z_2):Z_2$, $Z_{12}\times Z_2$, &1\\
							&$Z_3\times D_8$, $S_4$, $Z_2\times A_4$, $Z_2\times Z_2 \times S_3$, $Z_6\times Z_2\times Z_2$) & \\
							& $Z_4\times S_3$ ($D_{24}$, $Z_2\times (Z_3:Z_4)$, $(Z_6\times Z_2):Z_2$, $Z_{12}\times Z_2$, & 1\\
							& $Z_3\times D_8$,  $Z_2\times Z_2 \times S_3$, $Z_6\times Z_2\times Z_2$) &  \\
\hline
(26,13,12,6,13,2) & $D_{26}$ ($Z_{26}$) & 1  \\
\hline
(27,18,9,12,9,3) & $E_9:Z_3$, ($Z_9:Z_3$) & 2  \\
\hline\hline
\end{tabular} 
\caption{\footnotesize Proper divisible design Cayley graphs with $v \le 27$, $0 < \lambda_2 < k$, $\lambda_1 < k$}
\end{center} }
\end{table}

\bigskip

\noindent {\bf Acknowledgement} \\
This work has been fully supported by {\rm C}roatian Science Foundation under the projects 6732 and 5713.

\end{document}